\documentclass[reqno]{amsart}

\usepackage{amscd,amsthm,amsfonts,amssymb,amsmath}
\usepackage{amsmath,tikz-cd,hyperref}
\usepackage{color}
\hypersetup{
    colorlinks=true,
    citecolor=green,
    linkcolor=red,
  urlcolor=blue,
}

\usepackage{color}
\usepackage{hyperref}
\usepackage{color}
\hypersetup{
    colorlinks=true,
    citecolor=black,
    linkcolor=black,
  urlcolor=blue,
}

\usepackage{fullpage}
\usepackage[all]{xy}
\usepackage{mathtools}
\usepackage{mathrsfs,dsfont}
\usepackage{nccmath}
 
\makeatletter
\newcommand{\mylabel}[2]{#2\def\@currentlabel{#2}\label{#1}}
\makeatother

\newcommand{\Q}{\mathbb{Q}}
\newcommand{\Z}{\mathbb{Z}}

    \newcommand{\BA}{{\mathbb {A}}} 
    \newcommand{\BC}{{\mathbb {C}}}

    \newcommand{\BQ}{{\mathbb {Q}}} \newcommand{\BR}{{\mathbb {R}}}

     \newcommand{\BZ}{{\mathbb {Z}}}

    \newcommand{\CO}{{\mathcal {O}}}

 \newcommand{\SF}{{\mathscr {F}}}

    \newcommand{\fm}{{\mathfrak{m}}} 
     \newcommand{\fp}{{\mathfrak{p}}}

     \newcommand{\fD}{{\mathfrak{D}}}

    \newcommand{\cond}{\mathrm{cond^r}}

    \newcommand{\Gal}{{\mathrm{Gal}}} \newcommand{\GL}{{\mathrm{GL}}}
    
    \newcommand{\Hom}{{\mathrm{Hom}}}

    \newcommand{\ord}{{\mathrm{ord}}}

    \renewcommand{\mod}{\ \mathrm{mod}\ }

    \newcommand{\ov}{\overline}

    \theoremstyle{plain}
    \newtheorem{thm}{Theorem}[section] \newtheorem{cor}[thm]{Corollary}
    \newtheorem{lem}[thm]{Lemma}  \newtheorem{prop}[thm]{Proposition}
     \newtheorem{defn}[thm]{Definition}

\theoremstyle{remark} \newtheorem{remark}[thm]{Remark}
\theoremstyle{remark} 
\theoremstyle{remark} 

    \newcommand{\cO}{\mathcal O}
    \numberwithin{equation}{section}

\begin{document}
\title{The $p$-adic valuation of local resolvents, 
generalized Gauss sums and anticyclotomic Hecke $L$-values
of imaginary quadratic fields at inert primes}

\author{Ashay A. Burungale, Shinichi Kobayashi and Kazuto Ota}

\address{Ashay A. Burungale:  Department of Mathematics, The university of Texas at Austin,
2515 Speedway, Austin TX 78712} 
\email{ashayburungale@gmail.com}

\address{Shinichi Kobayashi: Faculty of Mathematics,
Kyushu University, 744, Motooka, Nishi-ku, Fukuoka, 819-0395, Japan.}
\email{kobayashi@math.kyushu-u.ac.jp}

\address{Kazuto Ota: \textsc{Department of Mathematics, Graduate School of Science, Osaka University Toyonaka, Osaka 560-0043, Japan}} 
\email{
kazutoota@math.sci.osaka-u.ac.jp}
\maketitle
\begin{abstract} 
We 
prove an asymptotic formula for
the $p$-adic valuation of 
Hecke $L$-values of an imaginary quadratic field at an inert prime $p$ along the anticyclotomic $\BZ_p$-tower. 
The key is determination of the $p$-adic valuation of generalized Gauss sums 
defined using Coates-Wiles homomorphism, 
 and  
of local resolvents in $\BZ_p$-extensions. This answers a question of Rubin. 
\end{abstract}

\tableofcontents
\section{Introduction}\label{Int}

$L$-functions bear an affinity to arithmetic. 
The $p$-adic valuation of a (normalized) $L$-value conjecturally encodes the size of Bloch-Kato Selmer group and Tate--Shafarevich group, invariants of the associated $p$-adic Galois representation. 
The $p$-divisibility properties of 
 $L$-values in a $p$-adic family 
 of motives 
is elemental to the arithmetic nature of $L$-values
 and 
 Iwasawa theory.
They reflect the underlying global arithmetic as well as local Perrin-Riou theory of the exponential map for the family, 
the latter mirroring 
variation of the integral structure of Bloch--Kato local subgroups over the family. 
When $p$ is a prime of ordinary reduction, general principles of Iwasawa theory predict a systematic variation of 
the $p$-part  of $L$-values. 
On the other hand, non-ordinary primes are still not well-understood, and 
the conjectural framework excludes basic examples such as anticyclotomic deformation of a CM elliptic curve at inert primes. 

In this paper we determine the $p$-adic valuation of 
central $L$-values of anticyclotomic deformation of a self-dual Hecke character of an imaginary quadratic field at an inert prime $p$ (cf.~Theorem \ref{theorem, intro main}). 
The investigation was first suggested by Rubin \cite{Ru} in the late 80's when he
proposed a framework for anticyclotomic CM Iwasawa theory at inert primes and made a 
conjecture on the structure of local units along a twist of the anticyclotomic direction.  
The recent proof of Rubin's conjecture \cite{BKO} has initiated progress towards the anticyclotomic CM Iwasawa theory (cf.~\cite{BKOb},~\cite{BKOs},~\cite{BKOY}), of which this work is a continuation.

Let $K$ be an imaginary quadratic field and $\eta_K$ the associated quadratic character of $\BQ$. Let $\varphi$ be a conjugate self-dual symplectic Hecke character of $K$ of infinity type $(1,0)$, that is,  
$$
  \varphi_\infty(z)=z^{-1} \text{ and $\varphi^*:=\varphi |\cdot|_{\BA_K^\times}^{1/2}$ satisfies {$\varphi^*|_{\BA^\times}=\eta_{K}$, }}
$$
where $\varphi_{\infty}: (K\otimes_{\BQ}\BR)^{\times} \to \BC^{\times}$ is the component of $\varphi$ at the infinite place and we regard  $K\otimes_{\BQ}\BR=\BC$, fixing an embedding $K\hookrightarrow {\BC}.$ 
Let $p$ be an odd prime inert in $K$. 
Let $K_\infty$ be the anticyclotomic $\BZ_p$-extension of $K$. 
We consider 
finite order Hecke characters $\chi$ of $K$ factoring through $K_\infty/K$. 
For a CM period $\Omega$ of $K$, the 
$L$-value $$\frac{L(\varphi \chi, 1)}{\Omega}$$ is algebraic and 
a basic question is to study its $p$-adic valuation under a fixed embedding 
$\overline{\Q} \hookrightarrow \mathbb{C}_p$. 
Let $v_{p}$ be the $p$-adic valuation of $\BC_p$ normalized as $v_{p}(p)=1$.
In the inert case, Greenberg \cite{Gr} found the interesting root number formula 
$$W(\varphi \chi)=W(\varphi)\cdot (-1)^{n-1}$$ 
for 
$\chi$
a finite anticyclotomic Hecke character of order $p^n>1$. 
In particular, if $n$ satisfies $(-1)^n=W(\varphi)$, then
$L(\varphi\chi, 1)=0$. 
To consider $v_{p}(\frac{L(\varphi\chi,1)}{\Omega})$, 
one may thus assume $(-1)^{n-1}=W(\varphi)$.

\subsubsection*{Results}
Our main result is the following.
\begin{thm}\label{theorem, intro main}
Let $E$ be a CM elliptic curve defined over $\BQ$ of conductor $N$ 
and $\varphi_{E}$ the associated Hecke character of the CM field $K$. 
Let $p\nmid 6N$ be a prime that is inert in $K$. 
Then there exist non-negative integers $\lambda$ and $\mu$ such that for any sufficiently large $n$ 
with $\varepsilon:=W(\varphi_{E})=(-1)^{n-1}$, 
we have 
\begin{equation}\label{equation, intro main}
v_p\left(\frac{L(\varphi_{E}\chi, 1)}{\Omega}\right)=\frac{\lambda}{p^{n-1}(p-1)}+\mu-\frac{n+1}{2}+
\frac{1}{p^{n-1}(p-1)}\left(\frac{1-\varepsilon}{2}+\sum_{k\equiv n-1\!\!\mod \!2 } (p^k-p^{k-1})\right)
\end{equation}
where $\chi$ is an anticyclotomic character of order $p^n$ 
 and 
the index $k$ runs through integers  $1 \leq k \leq n-1$ with the same parity as 
$n-1$.  

Moreover, if $p\nmid \frac{L(E_{/\BQ}, 1)}{\Omega}$, then 
\[
v_p\left(\frac{L(\varphi_{E}\chi, 1)}{\Omega}\right)=-\frac{n+1}{2}+
\frac{1}{p^{n-1}(p-1)}\sum_{k=1}^{\frac{n-1}{2}} (p^{2k}-p^{2k-1})
\] 
order $p^n$. 
\end{thm}

The main text considers more general self-dual Hecke characters $\varphi$ (cf.~Theorem~\ref{theorem, main L-value}).

Our formula \eqref{equation, intro main} 
is
essentially 
the same as 
Pollack's formula \cite{Po}
for the $p$-adic valuation of $L$-values of 
the cyclotomic deformation of an elliptic curve over $\BQ$ at a good supersingular prime $p$ (cf.~\cite{Nasy}, ~\cite[Prop.~6.9]{Po}).     
However, the arithmetic behind the formulas is very different. First, unlike the cyclotomic deformation, the anticyclotomic deformation is self-dual, accordingly Theorem \ref{theorem, intro main} concerns $\chi$ of $p$-power order of a fixed parity while the results of \cite{Po} 
apply to 
any finite order $\chi$.
In {\it{loc. cit.}}
the contribution of even/odd growth factor 
on the right-hand side is related 
to the Tate-Shafarevich group, 
whereas in 
the anticyclotomic case 
it comes from 
the Mordell-Weil group 
(cf.~\cite{AH0},~\cite{BKOb},~\cite{BKOs},
~\cite{Ko0}).
For the cyclotomic deformation, 
the summand $\frac{n+1}{2}$ on the right-hand side corresponds to the $p$-adic valuation of the Gauss sum $\tau(\chi)$. 
On the other hand, in our case
it is linked with 
a local resolvent (cf.~Theorem~\ref{thm, lr}) and 
 a generalized Gauss sum \eqref{delta-chi} defined 
by evaluation of Coates-Wiles logarithmic derivative at local units in the self-dual direction 
(cf.~Theorem~\ref{theorem, delta valuation}).

An application of the proof of Theorem \ref{theorem, intro main} and the main result of \cite{BKOb} is the following (cf.~Corollary~\ref{cor, main L-value}).

\begin{thm}\label{thm, intro main application}
Let $E$ be a CM elliptic curve defined over $\BQ$ of conductor $N$ 
and $\varphi_{E}$ the associated Hecke character of the CM field $K$. 
Let $p\nmid 6N$ be a prime that is inert in $K$.
Suppose that the root number of $E$ over $\BQ$ is $-1$. 
\begin{itemize}
\item[i)] We have  
\[
v_p\left(\frac{L(\varphi_E\chi, 1)}{\Omega}\right)\geq -\frac{3}{2}+\frac{1}{p-1}
\]
for any 
anticyclotomic 
character $\chi$ of $K$ of order $p^2$. 
(Note that $W(\varphi_E \chi)=+1$.) 
\item[ii)] If the equality holds in i) for some $\chi$  of order $p^2$, 
then 
$$\mathrm{ord}_{s=1} L(E_{/\Q}, s)=1.$$ 
In particular, the Tate-Shafarevich group of $E_{/\Q}$ is finite and the Mordell-Weil rank of $E(\Q)$ is $1$.  
\item[iii)] Conversely, suppose that  $\mathrm{ord}_{s=1} L(E_{/\Q}, s)=1$.  
Suppose also that $E(\Q)$ is dense in $E(\Q_p)\otimes_{\BZ} \BZ_{(p)}$, i.e. $E(\Q)\not\subset pE(\Q_p)\otimes_{\BZ} \BZ_{(p)}$, and 
$$\frac{L'(E_{/\Q}, 1)}{\Omega \cdot \mathrm{Reg}_E}$$ is a $p$-adic unit.
 Then the equality holds in i). 
In fact   
(\ref{equation, intro main}) holds with $\lambda=\mu=0$  
for all non-trivial $\chi$ of even $p$-power order.
\end{itemize}
\end{thm}
\begin{remark}
\begin{itemize}
\noindent
\item[\tiny{$\bullet$}] In the sequel \cite{BHKO}, we prove that the invariant $\mu$ appearing in \eqref{equation, intro main} vanishes. 
\item[\tiny{$\bullet$}] For primes $p$ split in $K$, an analogue of Theorem \ref{theorem, intro main} goes back to Katz \cite{Ka}, and of Theorem \ref{thm, intro main application} to Rubin~\cite{Ru1}.
\item[\tiny{$\bullet$}] The companion paper \cite{BKOs} considers variation of the associated Tate--Shafarevich groups (cf.~\cite{BF}). 
\item[\tiny{$\bullet$}] Finis studied
the $p$-adic valuation of Hecke $L$-values of an imaginary quadratic field in anticyclotomic families  (cf.~\cite{Fin},~\cite{Fin1},~\cite{Hi}).  
When $p$ splits, he determined the $p$-adic valuation 
for generic Hecke characters, however, 
for inert $p$ 
his results only apply to  
Hecke characters 
of  infinity type $(1,0)$ and conductor 
prime to $p$. 
The above results 
treat a complementary case (see also \cite{BHKO}). 
\end{itemize}

\end{remark}

\subsubsection*{About the proof}
We approach Theorem \ref{theorem, intro main}
as follows. 

A salient feature of Rubin's supersingular Iwasawa theory is the existence of a bounded 
$p$-adic $L$-function $\mathscr{L}_{p, \varphi}$ in  
the Iwasawa algebra $\mathcal{O}[\![\mathrm{Gal}(K_\infty/K)]\!]$. It depends on choice of a basis $v_\varepsilon$ of the module $V_\infty^{*, \varepsilon}$ of twisted local units \eqref{def:loc} in the anticyclotomic $\BZ_p$-extension $\Psi_\infty$ of the unramified quadratic extension $\Phi$ of $\BQ_p$, 
where $\varepsilon$ denotes the sign of $W(\varphi)$. 
The underlying $L$-values are interpolated 
as 
\begin{equation}\label{int}
\text{
$\mathscr{L}_{p, \varphi}(\chi)
=\frac{1}{\delta_{\chi^{-1}}(v_\varepsilon)}\cdot \frac{L(\overline{\varphi\chi}, 1)}{\Omega}$ 
}
\end{equation}
for anticyclotomic characters $\chi$ of order $p^n>1$ 
satisfying $(-1)^{n-1}=W(\varphi)$ 
(cf.~\cite{BKO}, ~\cite{Ru}). 
Here 
$\delta_\chi(v_\varepsilon)$ is a mysterious $p$-adic period factor \eqref{delta-chi}
analogous to the Gauss sum in the cyclotomic case, defined via 
Coates-Wiles homomorphism (or the dual exponential map). 
The $p$-adic valuation of $\mathscr{L}_{p, \varphi}(\chi)$ is controlled by the $\lambda$- and $\mu$-invariants of 
$\mathscr{L}_{p, \varphi}$. 
Hence 
it suffices to determine the valuation of $\delta_{\chi}(v_{\varepsilon})$. 
This local problem
was first suggested by Rubin \cite[pp.~421]{Ru}.

The $p$-adic period $\delta_\chi(v_\varepsilon)$ seems opaque. 
Its non-vanishing, being implicit in \eqref{int}, 
relies on Rubin's conjecture, 
which asserts the decomposition of twisted local units along $\Psi_\infty$ such that 
$V_\infty^{*}=V_\infty^{*,+}\oplus V_\infty^{*, -}$.  (cf. Theorem \ref{rubin-conj}). 
To study its valuation, we first 
build on the proof of Rubin's conjecture, 
leading to a system of local points ancillary to the underlying supersingular Iwasawa theory (cf.~Section~\ref{section: coates-wiles}). 
Then using the system, 
we relate 
 the valuation of $\delta_\chi(v_\varepsilon)$ to that of 
a Gauss-like sum
\[
\langle \alpha|\chi \rangle=\sum_{\sigma \in \mathrm{Gal}(\Psi_n/\Phi)} \chi(\sigma) \alpha^\sigma 
\]
for $\Psi_n$ the $n$-th layer of $\Psi_\infty$
and 
$\alpha \in \mathcal{O}_{\Psi_n}$ (cf.~Sections~\ref{ram} and~\ref{section: coates-wiles}). 
The even/ odd growth factor in \eqref{equation, intro main} originates from this connection between the valuation of $\delta_{\chi}(v_{\varepsilon})$ and $\langle \chi| \alpha \rangle$ (see also \eqref{delta-log}). The connection is indicative of a Perrin-Riou and Mellin transform theory along $\Psi_\infty$. 

The invariant 
$\langle \alpha |\chi \rangle$
is a 
primary  object in Galois module theory, often referred to as the local resolvent  (cf.~\cite{Fro}, ~\cite{FT}). 
Unlike the Gauss sum, it is 
inexplicit in general. 
An insight of this paper is 
its link with ramification theory leading to: 
\begin{thm}\label{thm, lr}
We have 
$$ v_{p}(\langle \alpha|\chi \rangle)
\geq \frac{n+1}{2}$$ for any 
character $\chi$ of $\Gal(\Psi_n/\Psi)$ of order $p^n>1$ and 
$\alpha \in \cO_{\Psi_{n}}$. 
Moreover, 
the equality holds 
if $\alpha$ is a uniformizer. 
\end{thm}

The answer to Rubin's question is then given by 
Theorem \ref{theorem, delta valuation}, and the proof concludes.

\subsubsection*{Vistas} 
The local resolvent $\langle \alpha |\chi \rangle$, the projector to $\chi$-part, is a basic 
object, and its valuation 
is of interest in broad context. A natural question is to link Theorem \ref{thm, lr} and the generalized Gauss sum $\delta_{\chi}(v_{\varepsilon})$ to Galois module theory  
 (cf.~\cite{Br}).

Anticyclotomic Iwasawa theory at 
inert primes 
complements the conjectural backdrop of global as well as local Iwasawa theory. 
Several of the foundational results in local Iwasawa theory are obtained by concrete calculations involving 
an explicit system of uniformizers along the underlying Iwasawa extension, such as the cyclotomic $\BZ_p$-extension. 
For example, Perrin-Riou theory and $(\varphi,\Gamma)$-theory for the cyclotomic deformation essentially rely on the system of cyclotomic units. 
Our study suggests that ramification theory may hold the key to replacing such explicit calculations. 
It was also employed in Tate's seminal work on $p$-divisible groups, leading to the notion of Tate trace, which is ancillary to the $(\varphi,\Gamma)$-theory (cf.~\cite{CG}, ~\cite{Ta}).

\subsubsection*{Acknowledgements}
We thank Mahesh Kakde, Shilin Lai,  
Georgios Pappas 
and Christopher Skinner for helpful discussions. We also thank the referee for instructive suggestions. 
 
The authors are grateful to Karl Rubin for his inspirational question.

This work was partially supported by the NSF grants 
DMS 2001409 and 2302064, 
and the JSPS KAKENHI grants  JP17H02836, JP18H05233, JP22H00096,  
JP17K14173, JP21K13774.

\section{The $p$-adic valuation of local resolvents in $\BZ_p$-extensions}\label{Gauss}
This section determines the valuation of local resolvents in totally ramified cyclic extensions. 
The main result is Theorem \ref{thm, lower bound gauss}, see also its consequences Corollary~\ref{cor, gauss valuation} and Theorem \ref{theorem, Z_p gauss valuation}.
\subsection{The set-up} Let $p$ be an odd prime. Fix an algebraic closure $\overline{\BQ}_p$ of ${\BQ}_p$. 

Let $K$ be a finite extension of $\BQ_p$, 
$\pi$ a uniformizer and $k$ the residue field.
Let $v_\pi$ be the valuation on $\ov{\BQ}_{p}$ normalized as $v_\pi(\pi)=1$. 
Let $L$ be a finite abelian extension of $K$ with Galois group $G$. 
For $\alpha \in \mathcal{O}_L$ and a character $\chi$ of $G$, define  
\[
\langle \alpha|\chi \rangle_{G}:=\sum_{\gamma \in G}    \chi(\gamma) \alpha^\gamma \;\in\;\overline{\mathbb{Q}}_p.
\]
For simplicity, we often denote $\langle \alpha|\chi \rangle_{G}$
by 
$\langle \alpha|\chi \rangle$.

The purpose of this section is to determine 
 the minimal $p$-adic valuation of $(\langle \alpha|\chi \rangle)_{\alpha \in \mathcal{O}_L}$  
 under the following two conditions: 
 
\begin{enumerate}

 \item[(\mylabel{item_Ram}{ram})] 
 The extension $L/K$ has the following type of upper ramification groups: 
\[
G=G^{-1}=G^0\supset G^1 \supsetneq \cdots \supsetneq G^n\supsetneq G^{n+1}=\{1\}
\]
for a non-negative integer $n$ so that 
$G^i/G^{i+1}$ is of order $p$ for $1\leq i\leq n$ and $G^0/G^1$ is of order $p-1$.
\item[(\mylabel{item_Cyc}{cyc})] $G^1$ is cyclic.
\end{enumerate}
For example, $\mathbb{\BQ}_p(\zeta_{p^{n+1}})$ over $\BQ_p$ satisfies these conditions. 
Note that  $\{0, 1, \dots, n\}$ is the jump sequence of the 
upper ramification groups by the Hasse-Arf theorem.  
Moreover, the existence of $L$ satisfying \eqref{item_Ram} for a sufficiently large $n$ implies that 
$K$ is unramified over $\BQ_p$ 
(cf. \cite{Maus}, \cite{Wy}). 

Let $K_m$ be the fixed field of $G^m$. In particular, 
$L=K_{n+1}$ and
 $K_1/K_0$ is a  tame extension of degree $p-1$. 
 Let $\varpi_m$ be a uniformizer of $K_{m}$. 
For simplicity, we denote the trace $\mathrm{Tr}_{K_{i+1}/K_i}$ by 
$\mathrm{Tr}_{i+1/i}$ and often the maximal ideal $\fm_{K_{i}} \subset \cO_{K_{i}}$ by $\fm_{i}$. 
We say $\chi$ is of conductor $p^{n+1}$ if  
$\chi |_{G^n}$ is non-trivial.

\subsubsection{Preliminaries}
\begin{lem}\label{lemma, lagrange}
For $\alpha, \beta \in \mathcal{O}_L$, we have
\begin{align*}
 \sum_{\gamma \in G}  \mathrm{Tr}_{L/K}(\alpha^\gamma\beta) \gamma
=\left(  \sum_{\gamma \in G}  \alpha^\gamma \gamma \right) 
\left(\sum_{\gamma \in G}  \beta^\gamma \gamma^{-1} \right).  
\end{align*}
\end{lem}
\begin{proof}
The assertion follows from 
 \begin{align*}
 \left(  \sum_{\gamma \in G}  \alpha^\gamma \gamma \right) 
\left(\sum_{\gamma \in G}  \beta^\gamma \gamma^{-1} \right)= \sum_{\gamma, \sigma \in G}
  \alpha^\gamma\beta^{\sigma} \gamma \sigma^{-1} 
= \sum_{\tau, \sigma \in G}  \alpha^{\tau\sigma}\beta^{\sigma} \tau.
\end{align*}
\end{proof}

By Lemma \ref{lemma, lagrange}, 
\begin{equation}\label{equation, gauss-trace}
\langle \alpha| \chi \rangle \langle \beta|\chi^{-1} \rangle=
 \sum_{\gamma \in G}  \mathrm{Tr}_{L/K}(\alpha^\gamma\beta) \chi(\gamma).
\end{equation}
We first investigate the $p$-adic valuation of the right-hand side.

\begin{lem}\label{lem, Herbrand}
Let $i$ be an integer 
such that $1\leq i \leq n+1$. 
\begin{itemize}
\item[i)] For $0\leq j \leq i$, the Herbrand function satisfies $$\psi_{K_{i}/K}(j)=p^j-1.$$  
In particular, for an integer $u$ such that $p^{j-1}\leq u \leq p^j-1$, we have 
$H_{u}=H^{j}$ where
 $H_u$ denotes the lower ramification group of $H:=G/G^{i}=\mathrm{Gal}(K_{i}/K)$.  
\item[ii)] For $0 \leq k \leq p-1$, we have $
\mathrm{Tr}_{i+1/i}\mathfrak{m}^k_{i+1}=\pi\mathcal{O}_{K_i}. $
\end{itemize}
\end{lem}
\begin{proof}
 i) The associated ramification group $H^{j}$ is $G^j/G^{i}=\mathrm{Gal}(K_{i}/K_j)$. 
 So 
\[
\psi_{K_{i}/K}(j)=\int_{0}^{j} (H^0: H^w)dw
=(K_1: K)(1+\cdots +p^{j-1})=p^j-1.
\] 
ii) By \cite[Ch.~V, ~\S3, Lem.~3 and Lem.~4]{Se}, 
we have 
 $$\mathrm{Tr}_{i+1/i}\mathfrak{m}^k_{i+1}=\mathfrak{m}_{i}^r$$ 
 for 
 $r=\lfloor (d+k)/p\rfloor$
where 
 $d$ 
 is  the exponent of the different of $K_{i+1}/K_i$ and the symbol $\lfloor x\rfloor$ denotes the largest integer $\leq x$.
 
 Let $\sigma$ be a generator of $\mathrm{Gal}(K_{i+1}/K_i)=\mathrm{Gal}(K_{i+1}/K)^i$. 
 Then by i) 
 the valuation of $\sigma \varpi_{i+1}-\varpi_{i+1}$ is $p^{i}$. 
Therefore, the exponent of the different generated by $\prod_{\sigma \in \mathrm{Gal}(K_{i+1}/K_i)\setminus\{e\}}(\sigma\varpi_{i+1}-\varpi_{i+1})$ is  
$d=(p-1)p^i$. 
In particular, $r=p^{i-1}(p-1)$ and so  $\mathrm{Tr}_{i+1/i}\mathfrak{m}^k_{i+1}=\mathfrak{m}_i^r=
\pi\mathcal{O}_{K_i}$. 
\end{proof}

\begin{remark}
Since the Herbrand function and the exponent of the different 
are determined 
solely 
by the upper ramification groups, 
it is sufficient to prove Lemma \ref{lem, Herbrand} for one extension of  
the same type of upper ramification groups. 
Therefore, 
we may assume 
$L=\mathbb{\BQ}_p(\zeta_{p^{n+1}})$ and 
prove Lemma \ref{lem, Herbrand} by direct calculations.  
\end{remark}

\begin{lem}\label{lem, trace 1}
 For $\sigma \in G^n$ and $\alpha, \beta \in \mathcal{O}_{L}$, 
   we have 
\[
\mathrm{Tr}_{n+1/0}\left((\sigma \alpha-\alpha)\beta\right) \in \mathfrak{m}_K^{n+1}.
\]
\end{lem}
\begin{proof}
Note that $G^n=G_{\psi(n)}=G_{p^n-1}$ by Lemma \ref{lem, Herbrand} i). 
Hence  
\[
\sigma \alpha-\alpha \in \frak{m}_{n+1}^{p^n}=\mathcal{O}_L\varpi_{1}.
\]
Therefore, by Lemma  \ref{lem, Herbrand} ii),  
\[
\mathrm{Tr}_{n+1/0} (\varpi_1\beta)
=\mathrm{Tr}_{1/0}(\varpi_1 \mathrm{Tr}_{n+1/1}\beta) 
\in \pi^n \mathrm{Tr}_{1/0} \frak{m}_1\subset \frak{m}^{n+1}_K. 
\]
\end{proof}

Take a set of representative $S$ of $G/G^n$ in $G$. Then 
\begin{align}
& \sum_{\gamma \in G}  \mathrm{Tr}_{n+1/0}(\alpha^\gamma\beta) \gamma=
 \sum_{\sigma \in S}  \sum_{\tau \in G^n}  \mathrm{Tr}_{n+1/0}(\alpha^{\sigma\tau}\beta)
 \sigma\tau\\
 &=\sum_{\sigma \in S}  \sum_{\tau \in G^n} 
 \mathrm{Tr}_{n+1/0}((\alpha^{\sigma\tau}-\alpha^\sigma)\beta)\sigma\tau+
 \sum_{\sigma \in S}   \mathrm{Tr}_{n+1/0}(\alpha^\sigma\beta)\sigma \sum_{\tau \in G^n} \tau 
 \label{equation, trace reduction}\\
& \equiv  \sum_{\sigma \in S}   \mathrm{Tr}_{n+1/0}(\alpha^\sigma\beta)\sigma 
\sum_{\tau \in G^n} \tau \mod \mathfrak{m}_K^{n+1}
\end{align}
by Lemma \ref{lem, trace 1}. 
If the conductor of $\chi$ is $p^{n+1}$, 
then $\sum_{\tau \in G^n} \chi(\tau)=0$. 
Therefore, part i) of the following theorem 
 is proved. 
\subsection{Main result}
\subsubsection{}
\begin{thm}\label{thm, lower bound gauss}
Let $K$ be a $p$-adic local field and $\pi$ a uniformizer. 
Let $L$ be a finite abelian extension of $K$ satisfying the condition \eqref{item_Ram}. 
Let $\chi$ be a character of the Galois group $G=\Gal(L/K)$ of conductor $p^{n+1}$.
 \begin{itemize}
 \item[i)]  For any $\alpha, \beta \in \mathcal{O}_L$, we have 
 \[
\langle \alpha|\chi \rangle \langle \beta|\chi^{-1}\rangle=\sum_{\gamma \in G}  \mathrm{Tr}_{L/K}(\alpha^\gamma\beta) \chi(\gamma)
 \equiv 0 \mod \pi^{n+1}. 
\]
\item[ii)] Suppose that $n=0$. Then 
there exist $\alpha, \beta \in \mathcal{O}_L$ such that 
 \[
v_\pi(\langle \alpha|\chi \rangle)+v_\pi(\langle \beta|\chi^{-1}\rangle)=v_\pi\left( \sum_{\gamma \in G}  \mathrm{Tr}_{L/K}(\alpha^\gamma\beta) \chi(\gamma)\right)=1
\]
where the valuation $v_\pi$ on $\overline{\Q}_p$ is normalized as $v_\pi(\pi)=1$. 
\item[iii)] Suppose $n\geq1$. 
Assume that $G^1$ is cyclic and $K$ is unramified over $\BQ_p$. 
Write $\chi=\omega \psi$ 
for $\omega$ 
a character  factoring through the unique subgroup $\Delta$ of $G$ of order $p-1$ 
and $\psi$ 
of order $p^{n}$. 
For $\alpha \in \mathcal{O}_L$, 
put $$\alpha_\omega:=
\langle \alpha|\omega \rangle_{\Delta}=\sum_{\rho \in \Delta} \omega(\rho) \alpha^\rho.$$ 
 Then for any $\alpha \in \mathfrak{m}_{L}$ with $v_L(\alpha_\omega)<p$, 
  there exists $\beta \in \mathcal{O}_L$ such that 
 \[
v_\pi(\langle \alpha|\chi \rangle)+v_\pi(\langle \beta|\chi^{-1}\rangle)=v_\pi\left( \sum_{\gamma \in G}  \mathrm{Tr}_{L/K}(\alpha^\gamma\beta) \chi(\gamma)\right)=n+1. 
\]
\end{itemize}
\end{thm}

\begin{remark}
 If $\omega=1$, then any uniformizer $\alpha$ of $L^\Delta$ satisfies  the condition 
 $v_L(\alpha_\omega)<p$ in iii). 
 If $\omega\not=1$, there exists $\alpha$ satisfying the condition by ii).
\end{remark} 
The above theorem will be proven in \S\ref{subsection-proof-thm-gauss}. We first describe some of its consequences.

\begin{cor}\label{cor, gauss valuation}
Let $K$ be a $p$-adic local field and 
$L$ a finite abelian extension of $K$ satisfying the condition \eqref{item_Ram} and \eqref{item_Cyc}. 
Let $\chi$ be a character of the Galois group $G=\Gal(L/K)$ of order $p^{n}>1$.
Assume that $K$ is unramified over $\BQ_p$. 
Then we have $$v_{\pi}(\langle \alpha|\chi \rangle) \geq \frac{n+1}{2}$$ for 
any $\alpha\in\cO_L$.   
Moreover, 
the equality holds for 
any $\alpha \in \mathfrak{m}_{L}$ with $v_L(\alpha)<p$.
\end{cor}
\begin{proof}
 Let $\iota$ be an element of $\mathrm{Gal}(\overline{\BQ}_p/K)$ such that 
 $\iota(\zeta_{p^m})=\zeta_{p^m}^{-1}$ for all natural numbers $m$.  
 The existence follows from the fact that $K \cap \BQ_p(\zeta_{p^\infty})=\BQ_p$  since 
 $K$ is unramified over $\BQ_p$.
 Note that $\iota(\langle \alpha|\chi \rangle)=\langle \iota(\alpha)|\chi^{-1}\rangle$ and so 
 $$v_{\pi}(\langle \alpha|\chi \rangle)=v_{\pi}(\langle \iota(\alpha)|\chi^{-1}\rangle).$$ 
The desired inequality follows from  (\ref{equation, gauss-trace}) and Theorem \ref{thm, lower bound gauss} i)  
with $\beta= \iota(\alpha)$. 

We now let $\alpha \in \fm_L$ with $v_{L}(\alpha)<p$, and 
take (another) $\beta$ as in Theorem \ref{thm, lower bound gauss} iii). 
Since 
both the valuations $v_\pi(\langle \alpha|\chi \rangle)$ and $v_\pi(\langle \iota(\alpha)|\chi^{-1}\rangle)$ are greater than 
or equal to  $\frac{n+1}{2}$, 
it follows that 
$$v_\pi(\langle \alpha|\chi \rangle)=\frac{n+1}{2}.$$  
\end{proof}

\begin{remark}
Suppose that $p>3$ and $n=0$. Then there exists a non-trivial character $\omega$ of  conductor $p$  
and $\alpha$ such that 
$v_{\pi}(\langle \alpha|\omega \rangle)=\frac{1}{p-1}<\frac{1}{2}=\frac{n+1}{2}$. 
In fact since $L/K$ is tame, there exists $\alpha \in \mathcal{O}_L$ such that $\mathcal{O}_L=\mathcal{O}_K[G]\alpha$. 
Then we have the character decomposition 
$\mathfrak{m}_L=\mathfrak{m}_K\oplus \bigoplus_{\omega\not=1} \mathcal{O}_K\langle \alpha|\omega \rangle$, 
and hence $v_{\pi}(\langle \alpha|\omega \rangle)=\frac{1}{p-1}$ for some $\omega$. 
\end{remark}

\begin{thm}\label{theorem, Z_p gauss valuation}
Let $\Psi$ be an unramified extension of $\BQ_p$ and $\pi$ a uniformizer. Let $\Psi_{\infty}/\Psi$ be a totally ramified $\BZ_p$-extension.
 Let $\chi$ be a finite character of $\mathrm{Gal}(\Psi_\infty/\Psi)$ of order $p^n>1$. 
 Then $$v_{\pi}(\langle \alpha|\chi \rangle_{\Gamma_n}) \geq \frac{n+1}{2}$$ for any $\alpha \in \mathcal{O}_{\Psi_n}$, where $\Gamma_n$ denotes the Galois group $\mathrm{Gal}(\Psi_n/\Psi)$ 
 for the $n$-th layer $\Psi_n$.  
 Moreover, 
the equality holds 
 for any uniformizer $\alpha$ of $\Psi_n$. 
\end{thm}
\begin{proof}
Let $K_1$ be a tamely ramified extension of $\Psi$ of degree $p-1$, 
and put $K=\Psi$ and $K_{n+1}=K_1 \Psi_{n}$.
The Galois group  $\Gamma_n$  has 
the upper ramification filtration 
\[
\Gamma_n=\Gamma^{-1}_n=\Gamma_n^0=\Gamma^1_n \supsetneq \cdots \supsetneq \Gamma_n^n\supsetneq \Gamma^{n+1}_n=\{1\}
\] 
with $\Gamma_n^i/\Gamma_n^{i+1}$ of order $p$ for $1\leq i \leq n$ (cf.~Proposition~\ref{prop, ramification groups of Z_p-ext}, see also \cite{Ha77},~\cite{Maus},~\cite{Wy}).
Since the upper ramification filtration is compatible with quotients, 
$L:=K_{n+1}$ satisfies the condition \eqref{item_Ram}. 

Hence, the assertion follows from Corollary \ref{cor, gauss valuation}. 
Note that 
$\langle \alpha|\chi\rangle_{G}=
\langle \mathrm{Tr}_{K_{n+1}/\Psi_n}\alpha|\chi \rangle_{\Gamma_{n}}$ for $\alpha \in K_{n+1}$, where $G:=\Gal(L/K)$.
\end{proof}
\begin{remark}
In particular, the above determines the valuation of the classical Gauss sum only based on upper ramification filtration.
\end{remark}
\subsubsection{Proof of Theorem \ref{thm, lower bound gauss}}\label{subsection-proof-thm-gauss}
Now we prove Theorem \ref{thm, lower bound gauss} ii), iii). 

By (\ref{equation, trace reduction}), for an appropriate choice of $\alpha$, 
it suffices to show 
the existence of $\beta$ such that 
\begin{equation}\label{equation, strict valuation}
v_\pi\left(\sum_{\sigma \in S}  \sum_{\tau \in G^n} 
 \mathrm{Tr}_{n+1/0}((\alpha^{\sigma\tau}-\alpha^\sigma)\beta)\chi(\sigma\tau)\right)=n+1. 
\end{equation}

\begin{prop}\label{prop, valuation cond=p}
Suppose that $n=0$ and let $\chi$ be a non-trivial character.  
Then there exist $\alpha, \beta \in \mathcal{O}_L$ satisfying (\ref{equation, strict valuation}).
In particular,  
Theorem \ref{thm, lower bound gauss} ii) holds. 
\end{prop}
\begin{proof}
In this case (\ref{equation, strict valuation}) 
simplifies to
\[
v_\pi\left(\sum_{\tau \in G} 
 \mathrm{Tr}_{1/0}((\alpha^{\tau}-\alpha)\beta)\chi(\tau)\right)=1.
 \] 
As before, $\mathrm{Tr}_{1/0}\mathcal{O}_L=\mathcal{O}_K$ and 
$\mathrm{Tr}_{1/0}\mathfrak{m}_1^i=(\pi)$ if $1\leq i \leq p-1$ (cf.~Lemma \ref{lem, Herbrand} ii)). 
Hence the pairing 
\[
\frak{m}_L/(\pi) \times \frak{m}_L/\frak{m}_L^{p-1} 
\rightarrow k, \quad (x, y) \mapsto 
\frac{1}{\pi}\mathrm{Tr}_{1/0}(xy) 
\]
of $k$-vector spaces 
is non-degenerate. 

Let $\alpha$ be such that $\mathcal{O}_L=\mathcal{O}_K[G]\alpha$.
Since $L/K$ is totally ramified, 
$(\alpha^{\sigma}-\alpha)_{\sigma\not=e \in G}$ 
is a basis of the $k$-vector space $\frak{m}_L/(\pi)$. 
Hence for a fixed $\sigma_0\not=e \in G$,  there exists a $k$-linear map 
$f: \frak{m}_L/(\pi) \rightarrow k$ such that $f(\alpha^{\sigma_0}-\alpha)=1$ and 
$f(\alpha^{\sigma}-\alpha)=0$ for $\sigma \not=\sigma_0$.  
By the non-degeneracy of the above pairing, there exists $\beta$ such that 
\[f(x)=\frac{1}{\pi}\mathrm{Tr}_{1/0}(x\beta).\] The assertion follows from this. 
\end{proof}

The rest of this section concerns the case $n>0$. 
We identify $G=\Delta \times G^1$. 
Assume that $G^1$ is cyclic and fix a generator $\gamma \in G^1$. 
Put $$S^1:=\{\gamma^i \;|\;0\leq i \leq p^{n-1}-1\}.$$ 
Then we take $S$ as $\Delta S^1:=\{\rho \sigma \;|\; \rho \in \Delta, \sigma \in S^1\}$. 
Let $\chi$ be a finite character of conductor $p^{n+1}$ and 
write $\chi=\omega \psi$ 
for $\omega$ 
a character  factoring through $\Delta$ and $\psi$ 
of order $p^{n}$. 
Then 
 \[
 \sum_{\sigma \in S}  \sum_{\tau \in G^n} 
 \mathrm{Tr}_{n+1/0}((\alpha^{\sigma\tau}-\alpha^\sigma)\beta)\chi(\sigma\tau)
 =\sum_{\sigma \in S^1}  \sum_{\tau \in G^n} 
 \mathrm{Tr}_{n+1/0}((\alpha^{\sigma\tau}_\omega-\alpha^\sigma_\omega)\beta)\psi(\sigma\tau). 
 \]

 Note that $\psi(\sigma \tau)-1$ is not a $p$-adic unit. So by Lemma \ref{lem, trace 1}, it 
 suffices to show
 
\begin{equation}\label{equation, strict valuation 2}
v_\pi\left(\sum_{\sigma \in S^1}  \sum_{\tau \in G^n} 
 \mathrm{Tr}_{n+1/0}((\alpha_\omega^{\sigma\tau}-\alpha_\omega^\sigma)\beta)\right)=n+1 
\end{equation}
 for some 
 $\beta$. 
 A key is the following.

\begin{prop}\label{prop, desired alpha}
Put 
\[
X:=\left\{x \in \mathfrak{m}_{L}\;| v_{L}(x)<p\right\}. 
\]
For any $\alpha \in \mathcal{O}_K+X$, there exists $\beta\in\cO_L$ such that  
\[ v_{\pi}\left(\sum_{\sigma \in S^1} 
 \mathrm{Tr}_{n+1/0}(\alpha^\sigma\beta)\right)=n. 
 \]
\end{prop}

We begin with a couple of preparatory lemmas.  
For $\beta \in \cO_{L}=\mathcal{O}_{K_{n+1}}$, consider the map 
\[
f_\beta: \mathcal{O}_{K_{n+1}} \;\rightarrow\; k, \qquad 
x \;\mapsto\; \pi^{-n} \mathrm{Tr}_{n+1/0}(\beta x) \mod \pi
\]
(cf.~Lemma \ref{lem, Herbrand} ii)). 
Put $\mathfrak{M}_1:=\mathcal{O}_{K_{n+1}}\mathfrak{m}_1
 =\frak{m}_{n+1}^{p^n}$.  
 Let us 
 denote the $k$-vector space $\mathcal{O}_{K_{n+1}}/\mathfrak{M}_1$ by $V$. 

\begin{lem}\label{lemma, non-degenerate}
 The map $f_\beta$ factors through $\mathcal{O}_{K_{n+1}}/\mathfrak{M}_1$. It
 is identically 
 zero if and only if  $\beta\in \mathfrak{M}_1$.  
  In particular, the pairing 
  \[ V \times V \rightarrow k, \quad (x, y) \mapsto 
  \pi^{-n} \mathrm{Tr}_{n+1/0}(xy) \mod \pi
  \]
  is non-degenerate. 
\end{lem}
\begin{proof}
By \cite[Ch.~V, ~\S3, Lem.~4]{Se}, we have
\[
\mathrm{Tr}_{i+1/i} \frak{m}_{i+1}^a=\frak{m}_i^{\lfloor\frac{a+p^{i}(p-1)}{p}\rfloor}. 
\]
In particular, 
\[
\mathrm{Tr}_{i+1/i} \frak{m}_{i+1}^{p^{b}}=\frak{m}_i^{p^{b-1}+p^{i-1}(p-1)}
=\pi\frak{m}_i^{p^{b-1}}. 
\]
Hence 
\[
\mathrm{Tr}_{n+1/0} \mathfrak{M}_1=
\mathrm{Tr}_{1/0}\mathrm{Tr}_{n+1/1} \frak{m}_{n+1}^{p^{n}}=\pi^n\mathrm{Tr}_{1/0}\frak{m}_1
=\pi^{n+1}\mathcal{O}_K. 
\]
Similarly, 
\[
\mathrm{Tr}_{i+1/i} \frak{m}_{i+1}^{p^b-1}=\pi\frak{m}_i^{p^{b-1}-1}
\]
and hence 
\[
\mathrm{Tr}_{n+1/0} \frak{m}_{n+1}^{p^{n}-1}
=\pi^{n}\mathcal{O}_K. 
\]
\end{proof}

Let $T$ be the 
$k$-linear operator on $V$ induced by $\gamma$.

\begin{lem}\label{lemma, N is nilpotent}
Put $N=T-1$. Then for $\alpha \in \mathcal{O}_K+X$,  we have 
\[
N^{p^{n-1}-1} \alpha \not=0, \quad N^{p^{n-1}}\alpha =0.
\]
\end{lem}
\begin{proof}
 Note that $N^{p^{n-1}}\varpi_{n+1}= (T^{p^{n-1}}-1)\varpi_{n+1}$ since $V$ is a $k$-vector space. 
By definition of the lower ramification groups, 
we have 
$\gamma\varpi_{n+1}=\varpi_{n+1}+u\varpi_{n+1}^p$ and 
 $\gamma^{p^{n-1}}\varpi_{n+1}=\varpi_{n+1}+v\varpi_{1}$ 
for $p$-adic units $u$, $v$.

Clearly, 
it suffices to prove the lemma for $\alpha\in X$. Pick $x\in X$ and write $v_{K_{n+1}}(x)=i<p$.
 By the previous paragraph, 
 $$(\gamma^{p^{n-1}}-1)x \in \varpi^{i-1}_{n+1}\mathfrak{M}_1
 \setminus  \varpi^{i}_{n+1}\mathfrak{M}_1$$ and 
 $(\gamma-1)\mathfrak{M}_1 \subset \varpi_{n+1}^p\mathfrak{M}_1$.  
 In particular, $N^{p^{n-1}}x=0$. 
 
 Suppose that $(\gamma-1)^{p^{n-1}-1} x \in \mathfrak{M}_1$.  
Then 
we have $(\gamma-1)^{p^{n-1}}x \in \varpi_{n+1}^p\mathfrak{M}_1$. 
This contradicts the fact that $(\gamma^{p^{n-1}}-1)x$ generates $\varpi^{i-1}_{n+1}\mathfrak{M}_1$. 
\end{proof}

\begin{proof}[Proof of Proposition \ref{prop, desired alpha}]
Let $\alpha$ be an element of $\mathcal{O}_K+X$.  
By Lemma \ref{lemma, N is nilpotent}, 
$\{N^{j}\alpha| 0\leq j \leq p^{n-1}-1\}$ is a linearly independent subset of $V$. 
In turn so is 
$\{T^j\alpha| 0 \leq j \leq p^{n-1}-1\}$. 
Hence there is a $k$-linear map $f: V \rightarrow k$ such that 
$f(\alpha)=1$ and $f(\alpha^{\gamma^j})=0$ 
for  $1\leq j \leq p^{n-1}-1$. 

In view of Lemma \ref{lemma, non-degenerate}, we may write 
\[
f(x)=\pi^{-n}\mathrm{Tr}_{n+1/0} (\beta x)
\] for some $\beta$. 
So 
\begin{align*}
  \pi^{-n} \sum_{\sigma \in S^1} 
\mathrm{Tr}_{n+1/0}(\alpha^\sigma\beta) = \sum_{j=0}^{p^{n-1}-1} 
 \pi^{-n}\mathrm{Tr}_{n+1/0}(\alpha^{ \gamma^j}\beta)\equiv 
\sum_{j=0}^{p^{n-1}-1}  f(\alpha^{\gamma^j})\equiv 1 \mod \pi.
\end{align*}
\end{proof}

We now return to Theorem \ref{thm, lower bound gauss}.
\begin{proof}[Proof of Theorem \ref{thm, lower bound gauss} iii)]
It is sufficient to find $\alpha$ and $\beta$ satisfying (\ref{equation, strict valuation 2}). 

First, consider the case $\omega=1$. We may assume that $\alpha \in \mathcal{O}_{L}^\Delta$ 
with $v_{L}(\alpha)<p$. 
Write $\mathrm{Tr}_{n+1/0}\alpha=p^n(p-1)a$ for $a \in \mathcal{O}_K$. 
Since $\langle \alpha| \chi \rangle=\langle \alpha-a|\chi \rangle$, 
we may 
replace $\alpha$ by 
$\alpha-a$.
 Then $\alpha \in  \mathcal{O}_K+X$ and 
$\mathrm{Tr}_{n+1/0}(\alpha)=0$. Take $\beta$ as in Proposition \ref{prop, desired alpha}. Then 
\begin{align*}
\sum_{\sigma \in S^1}  \sum_{\tau \in G^n} 
 \mathrm{Tr}_{n+1/0}((\alpha^{\sigma\tau}-\alpha^\sigma)\beta)
&=\sum_{\sigma \in S^1}  \sum_{\tau \in G^n} 
 \mathrm{Tr}_{n+1/0}(\alpha^{\sigma\tau}\beta)-p
 \sum_{\sigma \in S^1} 
 \mathrm{Tr}_{n+1/0}(\alpha^\sigma\beta)\\
 &=\mathrm{Tr}_{n+1/0}(\alpha) \mathrm{Tr}_{n+1/0}(\beta) 
  -p\sum_{\sigma \in S^1} 
 \mathrm{Tr}_{n+1/0}(\alpha^\sigma\beta)\\
 &=
  -p\sum_{\sigma \in S^1} 
 \mathrm{Tr}_{n+1/0}(\alpha^\sigma\beta). 
\end{align*}

Hence the assertion is a consequence of Proposition \ref{prop, desired alpha} and (\ref{equation, trace reduction}). 
(Note that if $n>1$, then the modification of $\alpha$ is inessential since 
$v_{\pi}(\mathrm{Tr}_{n+1/0}(\pi_{n+1}) \mathrm{Tr}_{n+1/0}(\beta))\geq 2n \geq n+2$.)

Now suppose that $\omega\not=1$. 
Then we have $\mathrm{Tr}_{n+1/0}(\alpha_\omega)=0$ and the assertion follows from the same argument 
as in the case $\omega=1$. 
\end{proof}

\section{The ramification group and uniformizers}\label{ram}

In this section we show 
the following existence of a system of 
uniformizers in a totally ramified $\BZ_p$-extension of an unramified field.

\begin{thm}\label{thm, frob uniformizer}
Let $p$ be an odd prime. Let $\Psi$ be an unramified extension of $\BQ_p$ 
with integer ring $\mathcal{O}$. 
Let $\Psi_{\infty}/\Psi$ be a totally ramified $\BZ_p$-extension 
and $R_n$ the integer ring of the $n$-th layer $\Psi_n$. 
Then there exists  a system of  uniformizers $(\pi_n)_n$ of $(R_n)_n$ 
such that 
\[
\pi_{n+1}^p\equiv \pi_n \mod pR_{n+1}.
\]
\end{thm}

We begin with a preliminary reduction. 

By local class field theory, $\Psi_\infty$ is contained in a Lubin-Tate extension of $\Psi$ arising from a uniformizer $\varpi$, which is universal norm for $\Psi_\infty$.
Put $\pi_0:=\varpi$
and pick a norm compatible sequence $(\pi_{n})_{n}$
for $\pi_n$ a uniformizer of $R_n$.
Since $\Psi_{n+1}/\Psi_n$ is totally ramified, there exists a monic Eisenstein polynomial 
$$f(x)=\sum_{i=0}^p a_ix^i \in R_{n}[x]$$ of degree $p$ 
such that $f(\pi_{n+1})=0$. 
Note that $a_{0}=-\pi_{n}$. 
To prove  Theorem \ref{thm, frob uniformizer}, 
it thus suffices to show that 
all but the constant and leading coefficients of  
$f(x)$ are divisible by $p$, 
i.e.
$f'(x) \in pR_{n+1}[x]$. 
Write $\fD_{n+1/n}$ for the different of $\Psi_{n+1}/\Psi_n$.

\begin{lem}
We have $f'(x) \in pR_{n+1}[x]$ if and only if $\mathfrak{D}_{n+1/n} \subset pR_{n+1}$. 
\end{lem}
\begin{proof}
 Note that  $\mathfrak{D}_{n+1/n}=(f'(\pi_{n+1}))$ and 
 $$\{\pi_{n+1}^{i}| 0 \leq i \leq p-1\}$$ is a basis of 
 the $R_n$-module $R_{n+1}$.  
 Since $f'(\pi_{n+1})=\sum_{i=1}^{p} ia_i \pi_{n+1}^{i-1}$, 
 it follows that $f'(\pi_{n+1}) \in pR$ if and only if $p | a_i$ for all $1\leq i \leq p-1$.
\end{proof}

Our approach relies on the following 
 (cf.~\cite[Prop.~3.3]{Ha77}).
\begin{prop}\label{prop, ramification groups of Z_p-ext}
 The upper ramification filtration of $\Gamma_n:=\mathrm{Gal}(\Psi_n/\Psi)$ is 
 given by 
\[
\Gamma_n=\Gamma_n^{-1}=\Gamma^0=\Gamma_n^1 \supsetneq \cdots \supsetneq \Gamma_n^n\supsetneq \Gamma_n^{n+1}=\{1\}
\] 
with $\Gamma^i_n/\Gamma^{i+1}_n$ of order $p$ for $1\leq i \leq n$.
\end{prop}

\begin{proof}[Proof of Theorem \ref{thm, frob uniformizer}]
By Proposition \ref{prop, ramification groups of Z_p-ext}, the gap sequence does not depend on 
the choice of the totally ramified $\mathbb{Z}_p$-extension $\Psi_\infty$.
Since the valuation of the different is determined by 
the gap sequence, it is also independent of the choice.
So
it suffices to check $\mathfrak{D}_{n+1/n} \subset pR_{n+1}$
for the cyclotomic $\BZ_p$-extension. 
Hence Theorem \ref{thm, frob uniformizer} follows from 
the case of the cyclotomic $\mathbb{Z}_p$-extension. 
\end{proof}

\begin{cor}\label{cor, frob uniformizer}
Let $\varpi_{n+1}$ be any uniformizer of $R_{n+1}$. 
Then $\varpi_{n+1}^p \in pR_{n+1}+R_n$. 
\end{cor}
\begin{proof}
Take $(\pi_m)_m$ to be a system of uniformizers as in Theorem \ref{thm, frob uniformizer}. 
 Write $\varpi_{n+1}=\sum_{i=1}^\infty a_i \pi_{n+1}^i$ 
 for $a_i \in \mathcal{O}$. 
 Then the assertion follows from Theorem \ref{thm, frob uniformizer}. 
\end{proof}

\section{The valuation of $\delta_{\chi}$}\label{section: coates-wiles}
This section 
determines the valuation of generalized Gauss sum $\delta_{\chi}$ 
(cf.~Theorem \ref{theorem, delta valuation}).

\subsection{The set-up} 
Let $p\ge 5$ be a prime. 
Let $\Phi $ be the unramified quadratic extension of $\Q_p$ 
and $\CO$ the integer ring.
Let $\mathscr{F}$ be a Lubin-Tate formal group  over $\mathcal{O}$ for   
the uniformizing parameter $\pi:=-p$. 
Let $\lambda$ denote the logarithm of $\mathscr{F}$. 

For $n\ge 0$,
write $\Phi_n=\Phi(\SF[\pi^{n+1}])$, the extension of $\Phi$ in $\BC_{p}$ generated by the $\pi^{n+1}$-torsion points of $\mathscr{F}$. 
Put $\Phi_{\infty}=\cup_{n\ge 0}\Phi_n$
and $T=T_{\pi}\SF$. 
Let $\Theta_\infty \subset \Phi_\infty$ be the $\BZ_p^2$-extension of $\Phi$, 
$\Psi_\infty$ the anticyclotomic $\BZ_p$-extension 
and $\Psi_n$ the $n$-th layer.
Put 
$\Gamma=\Gal(\Psi_{\infty}/\Phi)\cong\BZ_{p}$, $\Lambda=\cO[\![\Gamma]\!]$ 
and fix  a topological generator $\gamma$ of $\Gamma$. 
Let $U_n$ be the group of principal units in $\Phi_n$, that is, 
the group of elements in $\CO_{\Phi_n}^{\times}$ congruent to one modulo the maximal ideal. 
Put
 \[
T^{\otimes -1}=\Hom_{\CO}(T, \CO),\quad   V_{\infty}^{*} = \left(\varprojlim_nU_{n}\otimes_{\Z_p} T^{\otimes -1}\right)^{\Delta}\otimes_{\CO[\![\Gal(\Phi_{\infty}/\Phi )]\!]}\Lambda,
 \]
 where $\Delta:=\Gal(\Phi_\infty/\Theta_\infty)$  
 and the superscript $\Delta$ refers to $\Delta$-invariants.

Now we recall the Coates-Wiles logarithmic derivatives
\[
\delta: \varprojlim_{n} U_{n}\otimes_{\BZ_{p}} T^{\otimes -1} \rightarrow \CO, \qquad 
\delta_n: \varprojlim_{n} U_{n}\otimes_{\BZ_{p}} T^{\otimes -1} \rightarrow \Phi_n.
\]
For an element $x \in \varprojlim_{n} U_{n}\otimes_{\BZ_{p}} T^{\otimes -1}$, write 
$x=u\otimes v^{\otimes -1}$ where 
$u=(u_n)_n \in \varprojlim_n \,U_n$ and a generator 
$v=(v_n)_n \in T_\pi\SF$ as an $\CO$-module. 
Then consider the Coleman power series $f \in \CO[\![X]\!]^\times$ such that 
$f(v_n)=u_n$ and define 
\[
\delta(x)=\frac{f'(0)}{f(0)}, \quad  \delta_n(x)=\frac{1}{\lambda'(v_n)}\frac{f'(v_n)}{f(v_n)}.
\]
These maps are well-defined and Galois equivariant. 
For a finite character  $\chi: \Gal(\Phi_\infty/\Phi) \rightarrow \overline{\BQ}_p^\times$
factoring through  $\Gal(\Phi_n/\Phi)$,
put
\begin{equation}\label{delta-chi}
\delta_\chi(x)=\frac{1}{\pi^{n+1}}\sum_{\gamma \in \Gal(\Phi_n/\Phi)} \chi(\gamma)\delta_n(x)^\gamma.
\end{equation}
(The definition does not depend on the choice of $n$.)
For any anticyclotomic $\chi$, 
$\delta_\chi$ 
defines a map on $V_\infty^*$. 
The aim of this section is to investigate  
the image of $\delta_{\chi}$.

Let $\Xi$ be the set of finite characters of $\Gamma$ and $\cond{\chi}$ denote the conductor of $\chi\in\Xi$.
Put
\[
\begin{split}
\Xi^+&={\{}\chi \in \Xi \text{ $|$} \text{ }\cond\chi \text{ is an even power of } p{\}},\\
\Xi^-&={\{}\chi \in \Xi \text{ $|$} \text{ } \cond\chi \text{ is an odd power of } p{\}}.
\end{split}
\]
Define
\begin{equation}\label{def:loc}
V^{*,\pm}_{\infty} := {\{} v \in V_{\infty}^* \text{ $|$} \text{ }
\delta_{\chi}(v)=0 \quad \text{for every }\chi \in \Xi^{\mp} {\}}.
\end{equation}
Rubin showed that $V^{*,\pm}_{\infty}$ is a free $\Lambda$-module of rank one (cf.\ \cite[Prop.\ 8.1]{Ru}).

The main result of \cite{BKO} is a proof of the following conjecture of Rubin (cf.~\cite[Conj.~2.2]{Ru}).
 \begin{thm}\label{rubin-conj}
 We have
 \[
V_{\infty}^*=V^{*,+}_{\infty} \oplus V^{*,-}_{\infty}. 
\]
\end{thm}

\subsection{Local points} This subsection introduces a system of local points $c_n^{\pm}$ of $\mathscr{F}$, which leads to a link between the image of $\delta_{\chi}$ and local resolvents.
\subsubsection{Local cohomology}\label{s-lcg} 
We first 
describe certain cohomology groups related to  local units.

The Kummer map 
gives a natural isomorphism 
\begin{equation}\label{identification-kummer}
\varprojlim_n U_n \otimes T^{\otimes -1} \cong \varprojlim_n H^1(\Phi_n, \cO(1))\otimes T^{\otimes -1} \cong \varprojlim_n H^1(\Phi_n, T^{\otimes -1}(1))
\end{equation}
of $\CO[\![\Gal(\Phi_{\infty}/\Phi )]\!]$-modules. 
It induces a natural isomorphism 
\begin{equation}\label{anticyclotomic-identification}
V_{\infty}^*\cong \varprojlim_n H^1(\Psi_n, T^{\otimes -1}(1)), 
\end{equation}
and in turn
\begin{equation}\label{identification}
V_{\infty}^{*}/(\gamma^{p^n}-1)\cong H^1(\Psi_{n}, T^{\otimes -1}(1))
\end{equation}
(cf. \cite[Lem.~2.2]{BKOb}).

For a finite extension $L$ of $\Phi$, let 
\begin{equation*}\label{d-exp}
\exp^*_L: H^1(L,T^{\otimes -1}(1)) \to L
\end{equation*}
be the dual exponential map arising from the identification of 
 $\mathrm{coLie}(\mathscr{F})\otimes \Q_p$ with $\Phi$ so that 
   the invariant differential $d\lambda$ corresponds to $1$.

By the explicit reciprocity law of Wiles (cf.~\cite[Thm. 2.1.7, Ch. II]{K93},  \cite{Wi}),
 the diagram 
\begin{equation*}\label{delta-dual-exp}
\xymatrix{
\varprojlim_n U_n \otimes T^{\otimes -1}  \ar[r]\ar[d]_{{\pi^{-n-1}\delta_n}}& \varprojlim_n H^1(\Phi_n, T^{\otimes -1}(1))\ar[d]^{\exp^*_{\Phi_n}}\\
\Phi_n\ar[r]^{=} &\Phi_n 
}
\end{equation*}
commutes, where the upper horizontal map is (\ref{identification-kummer}).
Hence, 
for a character $\chi$ of $\Gal(\Psi_n/\Phi)$ and $v=(v_m)_{m\ge 0}  \in V_{\infty}^*=\varprojlim_m H^1(\Psi_m, T^{\otimes -1}(1))$ (cf.~(\ref{anticyclotomic-identification})), 
\begin{equation}\label{delta-dual-exoponential-with-chi}
\delta_{\chi}(v)= \sum_{\sigma \in \Gal(\Psi_n/\Phi)}\chi(\sigma)\exp^*_{\Psi_n}(v_n)^{\sigma}.
\end{equation}

\subsubsection{Local points}
We construct a system of local points relevant to local Iwaswa theory.

Fix a $\Lambda$-basis $v_\pm$ of $V_{\infty}^{*, \pm}$, which we often view as an element of $\varprojlim_n H^1(\Psi_n, T^{\otimes -1}(1))$ via (\ref{anticyclotomic-identification}).
For $n\ge 0$, put $\Lambda_n=\CO[\mathrm{Gal}(\Psi_n/\Phi)]$.
Let $v_{\pm, n}$ denote the image of $v_{\pm}$ in $H^1(\Psi_n, T^{\otimes -1}(1))$ via (\ref{identification}). 
By Theorem \ref{rubin-conj}, 
$\{v_{+, n}, v_{-,n}\}$ is a $\Lambda_n$-basis of $H^1(\Psi_n, T^{\otimes -1}(1))$.

 The local duality induces a natural pairing 
\[(\ \ , \ \ )_n: H^1(\Psi_n, T) \times H^1(\Psi_n, T^{\otimes -1}(1)) \to \CO.\]
Since it is perfect (cf.\ \cite[Lem.~2.3]{BKOb}),
so is
\begin{equation}\label{natural-perfect-pairing}
(\ \ , \ \ )_{\Lambda_n}: H^1(\Psi_n,T) \times H^1(\Psi_n, T^{\otimes -1}(1)) \to \Lambda_n,\quad (a,b) \mapsto \sum_{\sigma\in \Gal(\Psi_n/\Psi) }(a,b^{\sigma})_n\sigma, 
\end{equation}
which is also sesquilinear with respect to  the involution $\iota$ of $\Lambda_n$ arising from $\sigma \mapsto \sigma^{-1}$ for $\sigma \in \Gal(\Psi_n/\Phi).$
Let $\{v^{\perp}_{+, n}, v^{\perp}_{-,n}\} \subseteq H^1(\Psi_n,T)$ be the dual basis of $\{v_{-,n}, v_{+,n}\}$,  
that is,
\begin{equation}\label{char-of-perp}
\sum_{\sigma \in \Gal(\Psi_n/\Phi)}(v^{\perp}_{\pm, n}, v_{\pm, n}^{\sigma} )_n\sigma =0,\qquad \sum_{\sigma \in \Gal(\Psi_n/\Phi)}(v^{\perp}_{\pm, n}, v_{\mp, n}^{\sigma} )_n\sigma=1.
\end{equation}
Note that $v^{\perp}_{\pm, n}$ depends on the choice of $v_{\mp}$ but 
not of $v_{\pm}$.

Put
\[
\omega^+_n=\omega^+_n(\gamma)=\prod_{1\le k \le n,\ k: \mathrm{even}}\Phi_{p^k}(\gamma), \quad 
 \omega^-_n=\omega^-_n(\gamma)=(\gamma-1)\prod_{1\le k \le n,\ k: \mathrm{odd}}\Phi_{p^k}(\gamma) \in \BZ[\gamma] 
 \]
for $\Phi_{p^k}(X)$ the $p^k$-th cyclotomic polynomial. 
Also put $\omega^{+}_0=1$ and $\omega^{-}_0=\gamma-1$.
\begin{defn}[local points]
For $v_{\pm}$ and $\gamma$ as above, define
\[
c_n^{\pm}:=c_n^{\pm}(v_{\pm}, \gamma):=\omega^{\mp}_nv^{\perp}_{\pm,n} \in 
H^1(\Psi_n,T). \]
\end{defn}
In fact $c_{n}^\pm \in H^{1}_{\rm f}(\Psi_n,T)$ (cf.~\cite[Lem.2.5]{BKOb}) and so 
we may regard $c_{n}^\pm$ as an element of $\SF(\Psi_n)$ where 
$\SF(\Psi_{n}):=\SF(\fm_{n})$ for $\fm_{n}\subset R_n$ the maximal ideal. 

Let $\Xi_n^{\pm}$ denote the subset of $\Xi^{\pm}$ 
of characters factoring through
$\Gal(\Psi_n/\Phi)$ 
and put
\[
\SF(\Psi_n)^{\pm}=\{ x\in \SF(\Psi_n)  \ |  \lambda_{\chi}(x)=0 \text{ for } \chi \in \Xi^{\pm}_n   \} 
\]
for 
$$\lambda_{\chi}(x)=\frac{1}{p^{n}} \sum_{\sigma\in \Gal(\Psi_n/\Phi)}\chi^{-1}(\sigma)\lambda(x)^{\sigma}.$$

\begin{prop}\label{local-rational-points}\noindent
\begin{itemize}
\item[i)] We have 
\[
\SF(\Psi_n)^{}=\SF(\Psi_n)^{+}\oplus \SF(\Psi_n)^{-}.
\]
\item[ii)] For $n\ge 0$,  
the $\Lambda_n$-module $\mathscr{F}(\Psi_n)^\pm$ is generated by $c_n^{\pm} $. 
In particular, for $\varepsilon=(-1)^n$, the local point
$c_n^{\varepsilon}\in \mathscr{F}(\Psi_n)$ corresponds to a uniformizer of $\Psi_n$.
(Note that the underlying set of $\mathscr{F}(\Psi_n)$ is
 the maximal ideal of 
 the integer ring of $\Psi_n$.)
\end{itemize}
\end{prop}
\begin{proof}
i) This is \cite[Thm.~2.7~(1)]{BKOb}. ii) See \cite[Thm.~2.7]{BKOb} for the first assertion.
Then the latter just follows from i) and the fact that $\SF(\Psi_n)^{-\varepsilon} \subseteq \SF(\Psi_{n-1})$. 
\end{proof}
\begin{remark}
The above system of local points is elemental to anticyclotomic Iwasawa theory of CM as well as non-CM elliptic curves over imaginary quadratic fields with $p$ inert (cf.~\cite{BBL1},~\cite{BKOb}). 
\end{remark}

\subsection{Main result}

\begin{thm}\label{theorem, delta valuation}
Let $\chi$ be a finite character of $\mathrm{Gal}(\Psi_n/\Phi)$ of order $p^n >1$, and put $\varepsilon=(-1)^{n-1}$.  
Let $v_{\varepsilon}$ 
be a generator of $V_{\infty}^{*, \varepsilon}$.
Then we have 
\[
v_{p}(\delta_{\chi}(v_{\varepsilon}))=-\frac{n+1}{2}+
\frac{1}{p^{n-1}(p-1)}\left(\frac{1-\varepsilon}{2}+\sum_{(-1)^{k}=\varepsilon} (p^k-p^{k-1})\right)
\] 
where $k$ runs through integers between $1$ and $n-1$ such that $(-1)^k=\varepsilon$. 
\end{thm}
We begin with a preliminary. 
\begin{prop}\label{prop, gauss lambda}
Let $\alpha$ be a uniformizer of $\Psi_n$. 
Let $\chi$ be a character of $\mathrm{Gal}(\Psi_n/\Phi)$ of order $p^n>1$. 
Then 
\[
\langle \lambda(\alpha)|\chi \rangle -\langle \alpha|\chi \rangle \in p\langle \alpha|\chi \rangle R_n. 
\]
In particular, $v_p(\langle \lambda(\alpha)|\chi\rangle)=v_p(\langle \alpha|\chi \rangle)=\frac{n+1}{2}$. 
\end{prop}
\begin{proof} 
By an appropriate choice of the parameter $t$ of $\mathscr{F}$, 
we may 
 assume that 
 \[\lambda(t)=\sum_{i=0}^\infty (-1)^i \frac{t^{p^{2i}}}{p^i}.
 \] 
(For example this follows from Honda's theory of formal groups~\cite{Ho}.)
Let $\alpha$ be an arbitrary uniformizer of $\Psi_n$. 
Then write $\alpha^p \in pR_{n}+R_{n-1}$ by Corollary \ref{cor, frob uniformizer}. 

Hence we have 
\[
\lambda(\alpha)-\alpha=\sum_{i=1}^\infty (-1)^i \frac{\alpha^{p^{2i}}}{p^i} \in pR_n+\Psi_{n-1}. 
\]
Note that $\langle \beta|\chi \rangle=0$ if $\beta \in \Psi_{n-1}$ 
and $\langle \beta|\chi \rangle \in \langle \alpha|\chi \rangle R_{n}$ if $\beta \in R_n$. 
(By 
Theorem \ref{theorem, Z_p gauss valuation}, the valuation of $\langle \beta|\chi \rangle$ is minimum if $\beta$ is a uniformizer.)
The assertion follows.
\end{proof}
We now return to Theorem~\ref{theorem, delta valuation}.

\begin{proof}[Proof of Theorem~\ref{theorem, delta valuation}]
By definition,
\begin{equation}\label{cup}
\sum_{\sigma \in \Gal(\Psi_n/\Phi)}(c_n^{\pm}, v_{\mp,n}^{\sigma})_n\sigma=\omega_n^{\mp}  \in \Lambda_n. 
\end{equation}

Note that
\begin{align*}
\sum_{\sigma \in \Gal(\Psi_n/\Phi)}(c_n^{\pm}, v_{\mp,n}^{\sigma})_n\sigma 
&=\sum_{\sigma \in \Gal(\Psi_n/\Phi)}(\exp_{\Psi_n}(\lambda(c_n^{\pm})), v_{\mp,n}^{\sigma})_n\sigma \\
&=\sum_{\sigma \in \Gal(\Psi_n/\Phi)}\mathrm{Tr}_{\Psi_n/\Phi}\left(\lambda (c_n^{\pm}) \exp^*_{\Psi_n}(v_{\mp,n})^{\sigma}\right)\sigma\\
&=\left(
\sum_{\sigma \in \Gal(\Psi_n/\Phi)}\lambda(c_n^{\pm})^{\sigma} \sigma^{-1}\right)
\left(\sum_{\sigma \in \Gal(\Psi_n/\Phi)}\exp_{\Psi_n}^*(v_{\mp,n})^{\sigma} \sigma \right) 
\end{align*}
where $\exp_{\Psi_n}$ denotes the exponential map, the second equality follows from the fact that $\exp^*_{\Psi_n}$ is the dual of $\exp_{\Psi_n}$, and the third from Lemma \ref{lemma, lagrange}.
So in view of (\ref{delta-dual-exoponential-with-chi}) and (\ref{cup}),  the evaluation at a character  $\chi$ of $\Gal(\Psi_n/\Phi)$ gives
\begin{equation}\label{delta-log}
\langle \lambda_{}(c_n^{\pm})|\chi^{-1} \rangle \delta_{\chi}(v_{\mp})=\omega_n^{\mp}(\chi(\gamma)).
\end{equation} 
(In the case $(-1)^n= -\varepsilon$ the left and right hand sides of \eqref{delta-log} just vanish by definition.)

 For a primitive $p^n$-th root of unity $\zeta_{p^n}$ 
with $n>k$, we have 
\[v_p(\Phi_{p^k}(\zeta_{p^n}))=v_p((\zeta_{p^{n-k}}-1)/(\zeta_{p^{n-k+1}}-1))
=\frac{p^k-p^{k-1}}{p^{n-1}(p-1)}.\]
Hence the assertion is a consequence of Proposition \ref{local-rational-points} ii), Proposition \ref{prop, gauss lambda} and \eqref{delta-log}.
\end{proof}

\section{The $p$-adic valuation of Hecke $L$-values}\label{application}
This section presents an application of the $p$-adic valuation of generalized Gauss sum to that of Hecke $L$-values 
(cf. Theorem \ref{theorem, main L-value}).

Let $p\geq 5$ be a prime. Fix an algebraic closure $\ov{\BQ}$ and an embedding 
$\iota_{p}:\ov{\BQ}\hookrightarrow \ov{\BQ}_p$.

\subsection{Rubin's $p$-adic $L$-function}
Let $K$ be an imaginary quadratic field with $p$ inert and 
$H$ the Hilbert class field of $K$. 
Assume that 
\begin{equation}\label{clp}
p \nmid h_{K}.
\end{equation}

Let $K_\infty$ be the anticyclotomic $\BZ_p$-extension of $K$ 
and $K_n$ the $n$-th layer. In view of \eqref{clp} we often regard the set $\Xi$ of anticyclotomic $p$-power order  characters of $\Phi=K_p$
as that of anticyclotomic Hecke characters of $K$.

Let $\varphi$ be a Hecke character of $K$ of infinity type $(1,0)$. Let $E$ be a $\BQ$-curve in the sense of Gross \cite{Gr80}
such that the Hecke character $\varphi\circ N_{H/K}$  is associated to $E$, and $E$ has good reduction at each prime of $H$ above $p$. 
Let $\fp$ be the prime of $H$ above $p$ compatible with the embedding $\iota_p$.
Fix a Weierstrass model  of $E$ over $H \cap \mathcal{O}$ which is smooth at $\fp$. 
By considering a Galois conjugate of $E$ over $H$, 
we may assume 
the existence of a complex period $\Omega\in\BC^\times$
such that 
$L= \mathcal{O}_K\Omega$, where  
$L$ is the period lattice associated to the model. 

An insight of Rubin is the following existence of 
a $p$-adic $L$-function (cf.~\cite[\S10]{Ru},~\cite[\S6]{BKO}). 
\begin{thm}\label{theorem, Rubin-pL}
Let $\varepsilon \in \{+, -\}$ be the sign of the functional equation of the Hecke $L$-function 
$L(\varphi, s)$. Let $v_\varepsilon$ be a generator of the $\Lambda$-module $V_\infty^{*, \varepsilon}$. 
Then there exists $\mathscr{L}_p(\varphi, \Omega, v_\varepsilon)=:\mathscr{L}_{E}  \in \Lambda$ such that 
\[
\mathscr{L}_{E}(\chi)=\frac{1}{\delta_{\chi^{-1}}(v_\varepsilon)}\cdot \frac{L(\overline{\varphi\chi}, 1)}{\Omega}
\]
for $\chi\in\Xi^{\varepsilon}$.
\end{thm}

\noindent Here the non-vanishing of $\delta_{\chi}(v_{\varepsilon})$ is a consequence of Rubin's conjecture (cf.~\cite[Lem.~10.1]{Ru}).

A main result of \cite{BKOb} is the following 
$p$-adic Beilinson formula 
(cf.~\cite[Thm.~1.1]{BKOb}). 

\begin{thm}\label{ThmA}
Let $E_{/\BQ}$ be a CM elliptic curve with root number $-1$ 
and $K$ the CM field. 
Let $p\geq5$ be a prime of good supersingular reduction for $E_{/\BQ}$ and 
$\mathscr{L}_{E}$ the Rubin $p$-adic $L$-function. 
Then there exists a rational point $P\in E(\BQ)$ with the following properties. 
\begin{itemize}
\item[a)] We have
{
$$
\mathscr{L}_{E}(\mathds{1}) =\left(1+\frac{1}{p}\right) 
\frac{\mathrm{log}_{\omega}(P)^{2}}{\mathrm{log}_{\omega}(v_{-,0})}\cdot
 c_P
$$
for some $c_{P}\in\BQ^\times\mathcal{O}_K^\times$. 
 }
\item[b)]
$P$ is non-torsion if and only if $\ord_{s=1}L(E_{/\BQ},s)=1$. 
\item[c)]
If $\ord_{s=1}L(E_{/\BQ},s)=1$, then 
$$
c_{P}=\frac{L'(E_{/\BQ},1)}{\Omega \cdot \langle P, P \rangle_{\infty}}
$$
for $\langle \ , \ \rangle_{\infty}$ the N\'eron-Tate height pairing. 
\end{itemize}
\end{thm} 
\noindent Note that  $v_{-,0} \in E(\Phi)$ since $\mathds{1} \in\Xi^{+}$ and $\exp^*_E(v_{-,0})=0$ by definition, hence $\mathrm{log}_{\omega}(v_{-,0})$ is well-defined.

\subsection{Main result}
\begin{thm}\label{theorem, main L-value}
Let $\varphi$ be a self-dual Hecke character of an imaginary quadratic field $K$
of infinity type $(1, 0)$. Let $p$ be a prime inert in $K$ so that $p\nmid 6h_{K}\cond{\varphi}$. 
Then there exist non-negative integers $\lambda$ and $\mu$ such that for any sufficiently large integer $n$ with 
$\varepsilon:=W(\varphi)=(-1)^{n-1}$, we have 
\[
v_p\left(\frac{L(\varphi\chi, 1)}{\Omega}\right)=\frac{\lambda}{p^{n-1}(p-1)}+\mu-\frac{n+1}{2}+
\frac{1}{p^{n-1}(p-1)}\left(\frac{1-\varepsilon}{2}+\sum_{k\equiv n-1\!\!\mod \!2 } (p^k-p^{k-1})\right)
\] 
where $\chi$ is an anticyclotomic character 
 of order $p^n$  and 
the index $k$ runs through integers  $1 \leq k \leq n-1$ with the same parity as $n-1$.

Moreover, if $p\nmid \frac{L(\varphi, 1)}{\Omega}$, then 
\[
v_p\left(\frac{L(\varphi\chi, 1)}{\Omega}\right)=-\frac{n+1}{2}+
\frac{1}{p^{n-1}(p-1)}\sum_{k=1}^{\frac{n-1}{2}} (p^{2k}-p^{2k-1})
\] 
for all odd $n$  and any character $\chi$ 
of order $p^n$. 
\end{thm}
\begin{proof}
This is a simple consequence of Theorem \ref{theorem, Rubin-pL} and Theorem \ref{theorem, delta valuation}.
The integers $\lambda$ and $\mu$ are given as 
the $\lambda$- and $\mu$-invariants of Rubin's $p$-adic $L$-function. 
\end{proof}

\begin{cor}\label{cor, main L-value}
Let $E$ be a CM elliptic curve defined over $\Q$, and 
$\varphi_E$ the associated Hecke character of the CM field $K$. 
Let $p>3$ be a prime inert in $K$. 
Suppose that the root number of $E$ over $\BQ$ is $-1$ and $E$ has good reduction at $p$. 
\begin{itemize}
\item[i)] We have  
\[
v_p\left(\frac{L(\varphi_E\chi, 1)}{\Omega}\right)\geq -\frac{3}{2}+\frac{1}{p-1}
\]
for any anticyclotomic 
character $\chi$ of $K$ of order $p^2$. 
\item[ii)] If the equality holds  in i) for some $\chi$  of order $p^2$, 
then 
$$\mathrm{ord}_{s=1} L(E_{/\Q}, s)=1.$$ 
In particular, the Tate-Shafarevich group of $E_{/\Q}$ is finite and the Mordell-Weil rank of $E(\Q)$ is $1$.  
\item[iii)] Conversely, suppose that  $\mathrm{ord}_{s=1} L(E_{/\Q}, s)=1$. 
Suppose also that $E(\Q)$ is dense in $E(\Q_p)\otimes_{\BZ}\BZ_{(p)}$ and 
$$\frac{L'(E_{/\Q}, 1)}{\Omega \cdot \mathrm{Reg}_E}$$ is a $p$-adic unit. Then the equality holds in i). 
In fact   
(\ref{equation, intro main}) holds with $\lambda=\mu=0$  
for all non-trivial $\chi$ of even $p$-power order.
\end{itemize}

\end{cor}
\begin{proof}
The inequality directly follows from the proof of Theorem \ref{theorem, main L-value}. 
 The equality would imply that $\mathscr{L}_{E}(\chi)$ is a $p$-adic unit, and so
 the $p$-adic $L$-function $\mathscr{L}_{E}$ would itself be a unit of the Iwasawa algebra. 
 Hence the assertion follows from Theorem \ref{ThmA}.
\end{proof}


\begin{thebibliography}{XXXX} 




\bibitem{AH0} {A. Agboola} and B. Howard, \emph{Anticyclotomic Iwasawa theory of CM elliptic curves. II}, Math. Res. Lett. 12 (2005), no. 5-6, 611-621.










\bibitem{Br} J. Brinkhuis, \emph{On a comparison of Gauss sums with products of Lagrange resolvents}, Compositio Math. 93 (1994), no. 2, 155-170.

\bibitem{BBL1} A. Burungale, K. B\"uy\"ukboduk and A. Lei, {\it Anticyclotomic Iwasawa theory of $\GL_2$-type abelian varieties at non-ordinary primes}, Adv. Math. 439 (2024), Paper No. 109465, 63 pp. 



\bibitem{BF} A. Burungale and M. Flach, {\it The conjecture of Birch and Swinnerton-Dyer for certain elliptic curves with complex multiplication}, Camb. J. Math. 12 (2024), no. 2, 357--415.



\bibitem{BHKO} A. Burungale, W. He, S. Kobayashi and K. Ota, {\it Hecke $L$-values, definite Shimura sets and mod $\ell$ non-vanishing}, preprint, arXiv:2408.13932. 

\bibitem{BKO} A. Burungale, S. Kobayashi and K. Ota,
\textit{Rubin's conjecture on local units in the anticyclotomic tower at inert primes,}
Ann. of Math. (2) 194 (2021), no. 3, 943-966.



\bibitem{BKOb}A. Burungale, S. Kobayashi and K. Ota,
\textit{$p$-adic $L$-functions and rational points on CM elliptic curves at inert primes}, J. Inst. Math. Jussieu 23 (2024), no. 3, 1417--1460.

\bibitem{BKOs} A. Burungale, S. Kobayashi and K. Ota, {\it On the Tate-Shafarevich groups of CM elliptic curves over anticyclotomic $\Z_p$-extensions at inert primes}, 
Elliptic curves and modular forms in arithmetic geometry --Bertolini's 60th birthday conference proceedings, to appear.

\bibitem{BKOY} A. Burungale, S. Kobayashi, K. Ota and S. Yasuda, 
\emph{Kato's epsilon conjecture for anticyclotomic CM deformations at inert primes}, 
J. Number Theory 270 (2025), 17--67.  





 
\bibitem{CG} J. Coates and R. Greenberg, \emph{Kummer theory for abelian varieties over local fields}, 
Invent. Math. 124 (1996), no. 1-3, 129--174. 
 


\bibitem{Fin} T. Finis, \emph{Divisibility of anticyclotomic L-functions and theta functions with complex multiplication}, Ann. of Math. (2) 163 (2006), no. 3, 767-807.

\bibitem{Fin1} T. Finis, \emph{The $\mu$-invariant of anticyclotomic $L$-functions of imaginary quadratic fields}, J. Reine Angew. Math. 596 (2006), 131-152.

\bibitem{Fro} A. Fr\"ohlich, \emph{Galois module structure of algebraic integers}, Ergebnisse der Mathematik und ihrer Grenzgebiete (3) [Results in Mathematics and Related Areas (3)], 1. Springer-Verlag, Berlin, 1983. x+262 pp. 


\bibitem{FT} A. Fr\"ohlich and M. Taylor, \emph{The arithmetic theory of local Galois Gauss sums for tame characters}, Philos. Trans. Roy. Soc. London Ser. A 298 (1980/81), no. 1437, 141-181.

\bibitem{Gr} R. Greenberg, \emph{On the critical values of Hecke L-functions for imaginary quadratic fields}, 
Invent. Math. 79 (1985), no. 1, 79-94.




\bibitem{Gr80} B. Gross, \emph{Arithmetic on elliptic curves with complex multiplication}, With an appendix by B. Mazur. Lecture Notes in Mathematics, 776. Springer, Berlin, 1980. iii+95 pp. 


\bibitem{Ha77}
M. Hazewinkel, \emph{On norm maps for one dimensional formal groups. III.}, Duke Math. J. 44 (1977), no. 2, 305-314.



\bibitem{Hi} H. Hida, \emph{Non-vanishing modulo $p$ of Hecke $L$-values}, Geometric aspects of Dwork theory. Vol. I, II, 735-784, Walter de Gruyter, Berlin, 2004.

\bibitem{Ho} T. Honda, \emph{On the theory of commutative formal groups},  J. Math. Soc. Japan 22 (1970), 213-246.



\bibitem{K93} K. Kato,
\textit{Lectures on the approach to {I}wasawa theory for
              {H}asse-{W}eil {$L$}-functions via {$B_{\rm dR}$}. {I}},
              Arithmetic algebraic geometry ({T}rento, 1991),
              Lecture Notes in Math. 1553,
              50-163,
              Springer, Berlin, 1993.
\bibitem{Ka} N. Katz, \emph{$p$-adic interpolation of real analytic Eisenstein series},  Ann. of Math. (2) 104 (1976), no. 3, 459-571.



\bibitem{Ko0} S. Kobayashi, \emph{Iwasawa theory for elliptic curves at supersingular primes}, Invent. Math. 152 (2003), no. 1, 1-36. 





 
 \bibitem{Maus} 
 E. Maus, 
\emph{On the jumps in the series of ramifications groups}, 
Colloque de Th\'eorie des Nombres (Univ. Bordeaux, Bordeaux, 1969), pp. 127-133. 
Bull. Soc. Math. France, M\'em. 25, Soc. Math. France, Paris, 1971. 




\bibitem{Nasy}
A. G. Nasybullin, \emph{Elliptic curves with supersingular reduction over $\Gamma$-extensions} (Russian),  
Uspehi Mat. Nauk 32 (1977), no. 2(194), 221--222.


\bibitem{Po} R. Pollack, 
\emph{On the $p$-adic $L$-function of a modular form at a supersingular prime}, Duke Math. J. 118 (2003), no. 3, 523-558. 


 
 \bibitem{PRbook} B. Perrin-Riou,
 \emph{Fonctions $L$ $p$-adiques des repr\'esentations $p$-adiques}, 
 Ast\'erisque No. 229 (1995), 198 pp. 
 

\bibitem{Ru}K. Rubin, \emph{Local units, elliptic units, Heegner points and elliptic curves}, 
Invent. Math. 88 (1987), no. 2, 405-422.



\bibitem{Ru1} K. Rubin, \emph{$p$-adic L-functions and rational points on elliptic curves with complex multiplication}, Invent. Math. 107 (1992), no. 2, 323--350.





\bibitem{Se} J.-P. Serre, \emph{Local fields}, Translated from the French by Marvin Jay Greenberg. Graduate Texts in Mathematics, 67. Springer-Verlag, New York-Berlin, 1979. viii+241 pp.




 

 

 
 \bibitem{Ta} J. Tate, \emph{$p$-divisible groups}, 1967 Proc. Conf. Local Fields (Driebergen, 1966) pp. 158-183 Springer, Berlin.
 
  \bibitem{Wi} A. Wiles, \emph{Higher explicit reciprocity laws}, Ann. of Math. (2) 107 (1978), no. 2, 
  235-254.
   \bibitem{Wy}
 B. Wyman, 
 \emph{Wildly ramified gamma extensions}, Amer. J. Math. 91 (1969), 135-152.













 
  
     



\end{thebibliography}
\end{document}